\documentclass[11pt]{amsart}
\usepackage{amsmath,amssymb,amsfonts,eucal}
\usepackage{amsthm}
\usepackage{geometry}
\usepackage{array}
\usepackage{resizegather}
\usepackage{graphicx}
\usepackage{color}
\usepackage[normalem]{ulem}
\usepackage[latin1]{inputenc}

\theoremstyle{plain}
\newtheorem{thm}{Theorem}[section]

\newtheorem{lemma}[thm]{Lemma}

\theoremstyle{definition}

\newtheorem{remark}{Remark}[section]

\numberwithin{equation}{section}

\newcommand{\ds}{\displaystyle}

\newcommand{\dsum}{\ds\sum}

\newcommand{\eqskip}{ \vspace*{2mm}\\ }
\newcommand{\R}{\mathbb{R}}
\newcommand{\N}{\mathbb{N}}

\newcommand{\fr}[2]{\frac{\ds #1}{\ds #2}}

\newcommand{\cdr}{\mathcal{C}_\ell}

\makeatletter
\@namedef{subjclassname@2020}{%
  \textup{2020} Mathematics Subject Classification}
\makeatother

\begin{document}

\title{Families of non-tiling domains satisfying P\'{o}lya's conjecture}

\author[P. Freitas]{Pedro Freitas}
\author[I. Salavessa]{Isabel Salavessa}
\address{Departamento de Matem\'{a}tica, Instituto Superior T\'{e}cnico, Universidade de Lisboa, Av. Rovisco Pais, 1049-001 Lisboa, Portugal
\& Grupo de F\'{\i}sica Matem\'{a}tica, Faculdade de Ci\^{e}ncias, Universidade de Lisboa, Campo Grande, Edif\'{\i}cio C6,
1749-016 Lisboa, Portugal}
\email{pedrodefreitas@tecnico.ulisboa.pt}
\address{Grupo de F\'{\i}sica Matem\'{a}tica, Faculdade de Ci\^{e}ncias, Universidade de Lisboa, Campo Grande, Edif\'{\i}cio C6,
1749-016 Lisboa, Portugal \& Departamento de F\'{\i}sica, Instituto Superior T\'{e}cnico, Universidade de Lisboa, Av. Rovisco Pais, 1049-001 Lisboa, Portugal}
\email{isabel.salavessa@tecnico.ulisboa.pt}

\date{\today}

\begin{abstract} We show the existence of classes of non-tiling domains satisfying P\'{o}lya's conjecture in any dmension, in both the Euclidean and
non-Euclidean cases.
This is a consequence of a more general observation asserting that if a domain satisfies P\'{o}lya's conjecture eventually, that is, for a sufficiently
large order of the eigenvalues, and may be partioned into $p$ non-overlapping isometric sub-domains, with $p$ arbitrarily large, then there exists
an order $p_{0}$ such that for $p$ larger than $p_{0}$ all such sub-domains satisfy P\'{o}lya's conjecture. In particular, this allows us to show that families
of sectors of domains of revolution with analytic boundary, and thin cylinders satisfy P\'{o}lya's conjecture, for instance. We also improve upon the Li-Yau
constant for general cylinders in the Dirichlet case.
\end{abstract}
\maketitle
%\twocolumn

\section{Introduction}

It is now more than sixty years since P\'{o}lya conjectured that Dirichlet and Neumann Laplace eigenvalues of Euclidean domains are bounded from below
and above, respectively, by the first term in the corresponding Weyl asymptotics~\cite{poly1}. More precisely, given a bounded domain $\Omega\subset\R^{n}$
and denoting these eigenvalues by
\[
 \hspace*{2cm} 0<\lambda_{1}(\Omega)\leq \lambda_{2}(\Omega)\leq \dots  \to \infty, \mbox{ as } k\to \infty \hspace*{2cm}\mbox{ [Dirichlet]}
\]
and
\[
 \hspace*{2cm} 0=\mu_{1}(\Omega)\leq \mu_{2}(\Omega)\leq \dots  \to \infty, \mbox{ as } k\to \infty, \hspace*{2cm}\mbox{ [Neumann]}
\]
P\'{o}lya's conjecture states that
\begin{equation}\label{polconj}
 \mu_{k+1}(\Omega) \leq \fr{4\pi^2}{\omega _{n}^{2/n}} \left(\fr{k}{|\Omega|}\right)^{2/n} \leq \lambda_{k}(\Omega),
\end{equation}
where $|\Omega|$ and $\omega_{n}$ denote, respectively, the $n-$measure of $\Omega$ and of the $n-$ball of unit radius.
A few years later P\'{o}lya went on to prove that indeed this is the case for the Dirichlet problem for domains which tile the plane~\cite{poly2} -- see~\cite{urak83},
for the extension of P\'{o}lya's result to $\R^{n}$. P\'{o}lya also proved in~\cite{poly2} that the result
holds in the Neumann case, but only for the more restricted class of {\it regularly plane covering} domains, with the general case having been
proved later by Kellner~\cite{kelln}.

Although the conjecture remains open to this day in both cases, it is at the origin of a number of results in the literature along several different
directions. One of the first was what may be called a proof-of-concept result by Lieb in~1980~\cite{lieb}, who showed that for general bounded domains
there exists an absolute constant $C = C_{n}$, depending only on the dimension, such that
\begin{equation}\label{polyaineq}
  \lambda_{k}(\Omega) \geq C_{n} \times\fr{4\pi^2}{\omega _{n}^{2/n}}\left(\fr{k}{|\Omega|}\right)^{2/n}.
\end{equation}
Three years later, Li and Yau~\cite{liyau} improved this result by showing that the constant $C_{n}$ above could be taken to equal
$
 \fr{n}{n+2}.
$
A Neumann counterpart to this result
was obtained by Kr\"{o}ger in $1992$~\cite{kroger}, who showed that 
\begin{equation}\label{kroger}
\mu_{k}\leq \left(\fr{n+2}{2}\right)^{2/n}\fr{4\pi^2}{(\omega_n)^{2/n}}\left(\fr{k-1}{|\Omega|}\right)^{2/n}.
\end{equation}
Other results with different constants may be found, for instance, in Urakawa's work~\cite{urak83}, where the {\it lattice packing density} of
a domain is used to measure how close to one the factor multiplying the constant in the Weyl asymptotics can be. 

Another approach consists in identifying specific constructions that produce new domains satisfying~\eqref{polconj}. This was done by Laptev
in~\cite{lapt}, who obtained the first examples of non-tiling domains in dimesions four or higher. More generally, he considered $\Omega$ to
be the Cartesian product between two domains, say $\Omega = \Omega_{1}\times\Omega_{2}\subset\R^{p}\times\R^{q}$, such that $p$ is larger than or
equal to two and satisfies P\'{o}lya's conjecture in the Dirichlet case. Then the same holds for the Cartesian product $\Omega$ -- see also~\cite{hs} for
further results for Neumann boundary conditions. We note that the case of cylinders given by the Cartesian product of a general domain and an interval
is still open.

More recently, a connection between this problem and that of determining optimal spectral domains was found by Colbois and El Soufi, who
proved an equivalence between P\'{o}lya's conjecture and the convergence of the sequence $k^{-2/n}\lambda_{k}^{*}$, where $\lambda_{k}^{*}$
denotes the smallest possible value of $\lambda_{k}(\Omega)$ for domains $\Omega$ with given $|\Omega|$~\cite{colbsouf} -- see also~\cite{frlapa}
for a similar result for domains which are unions of scaled copies of a single domain, and for references to other similar problems.

In this note we explore the second approach described above, providing a way of obtaining families of domains satisfying P\'{o}lya's conjecture
but which do not tile the space, and without the need to assume the existence of other such domains a priori. Another case where such a
domain has been shown to exist is in the recent proof for the important (and iconic) case of Euclidean balls~\cite{flps}.

% Note that Laptev's construction requires the space to have dimention at least four in order to obtain a non-tiling domain
% satisfying~\eqref{polconj}.

The idea behind our approach is quite simple, and may be thought of as replicating P\'{o}lya's argument to the case where instead of tiling the whole
space we start from a domain that may be tiled by an arbitrarily large number of isometric copies of a subdomain. It is quite straightforward
to see that, should the larger domain satisfy P\'{o}lya's conjecture, then the same is true for the smaller domain~\cite{fms}. However, and under
certain conditions, we are able to show that it is sufficient for P\'{o}lya's conjecture to hold for large enough eigenvalues of the original domain,
for this to imply that there is an order of division of that domain into tiling subdomains, after which they must satisfy P\'{o}lya's conjecture.
Using this approach, we are able to show that, for instance, planar circular sectors with a sufficiently small angle opening and sufficiently thin
cylinders satisfy inequality~\eqref{polconj}, as do certain sufficiently thin tiling subsets of solids of revolution in $\R^{n}$. A related idea had
been used previously by Hersch~\cite{hers} and also by the first author of the current paper~\cite{frei1}, but only as a way of obtaining bounds for
the first eigenvalue.

In the next section we formalise these ideas and establish the basic lemma which then allows us to provide several examples of non-tiling
domains satisfying P\'{o}lya's conjecture. We then consider the case of general cylinders in more detail, for which we substantially improve the constant
$C_{n} = n/(n+2)$ in~\eqref{polyaineq}.

\section{The basic lemma}

In order to proceed, it is convenient to have the following two definitions.

Given domains $\Omega,\Omega_{0}\subset\R^{n}$ and an integer $p$, we say that $\Omega$ is $p-$tiled by $\Omega_{0}$ if $\Omega$ is the
interior of the closure of the union of $p$ nonoverlapping isometric copies of $\Omega_{0}$.

We say that $\Omega$ satisfies P\'{o}lya's conjecture eventually (for Dirichlet eigenvalues), if there is an order $k_{0}$ such that
$\lambda_{k}$ satisfies the right-hand side inequality in~\eqref{polconj} for all $k$ larger than or equal to $k_{0}$. A similar definition
applies to the Neumann eigenvalues and the left-hand side inequality in~\eqref{polconj}. A class of domains which fall into this category are
those having a second positive (resp. negative) term in the Weyl asymptotics, for the Dirichlet (resp. Neumann) eigenvalues respectively. More
precisely, under certain geometric conditions in $\Omega$ we have the following two-term Weyl asymptotics~\cite[Theorem~1.6.1 and Example~1.6.16]{sava}
\begin{eqnarray}\label{2termweyl}
\lambda_{k}(\Omega) = \frac{\ds 4\pi^2k^{2/n}}{\ds \left(\omega_{n}|\Omega|\right)^{2/n}}
 + \frac{\ds 2\pi^2\omega_{n-1}|\partial\Omega|k^{1/n}}{\ds n \left(\omega_{n}|\Omega|\right)^{1+1/n}}+ {\rm o}(k^{1/n})
\end{eqnarray}
and
\begin{eqnarray}
 \mu_{k}(\Omega) = \frac{\ds 4\pi^2k^{2/n}}{\ds \left(\omega_{n}|\Omega|\right)^{2/n}}
 - \frac{\ds 2\pi^2\omega_{n-1}|\partial\Omega|k^{1/n}}{\ds n \left(\omega_{n}|\Omega|\right)^{1+1/n}}+ {\rm o}(k^{1/n}),\label{2termweylNeumann}
\end{eqnarray}
where $|\partial\Omega|$ denotes the $(n-1)-$measure of the boundary of $\Omega$.

\begin{lemma}\label{mainlem}
 Let $\Omega\subset \R^{n}$ be a bounded domain satisfying P\'{o}lya's conjecture eventually and assume that for each $p\in\N$ there exists a domain $\Omega_{p}$ such
 that $\Omega_{p}$ $p-$tiles  $\Omega$. Then there exists $p_{0}$ such that $\Omega_{p}$ satisfies
 P\'{o}lya's conjecture for the Dirichlet eigenvalues for all $p$ larger than $p_{0}$. A similar conclusion holds for the Neumann eigenvalues.
\end{lemma}

\begin{proof} If $\Omega$ contains $p$ nonoverlapping copies of $\Omega_{p}$, it follows by Lemma~1 in~\cite{poly2} that
$\lambda_{kp}(\Omega) \leq \lambda_{k}\left(\Omega_{p}\right)$ for all positive integer $k$. Hence,
\[
 \lambda_{k}\left(\Omega_{p}\right) \geq \fr{4\pi^2}{\left(\omega_{n}|\Omega|\right)^{2/n}}\left(kp\right)^{2/n}
\]
for $kp$ sufficiently large, say $kp \geq k_{0}$. By taking $p\geq k_{0}$, we obtain that the above inequality is satisfied for all positive integer $k$. On
the other hand, $|\Omega| = p |\Omega_{p}|$, and so
\[
 \lambda_{k}\left(\Omega_{p}\right) \geq \fr{4\pi^2}{\left(\omega_{n}|\Omega_{p}|\right)^{2/n}}k^{2/n}
\]
for all $p\geq k_{0}$ and $k\in\N$, proving the result in the Dirichlet case.

Similarly, by Lemma~2 in~\cite{poly2} we have $\mu_{(k-1)p+1}(\Omega)\geq \mu_k(\Omega_p)$, for all $k\geq 1$, and assuming the Neumann eigenvalues
of $\Omega$ satisfy P\'{o}lya's inequality for $k\geq k'_0$, we obtain for $(k-1)p +1\geq k'_0$,
\begin{eqnarray*}
\mu_k(\Omega_p) &\leq&  \fr{4\pi^2}{(\omega_n|\Omega_p|)^{2/n}}(k-1)^{2/n}.
\end{eqnarray*}
If $k'_0=1$ we are done for any $p$ and $k\geq 1$. If $k'_0\geq 2$
 we  take $p\geq k'_0$ and  we see that P\'{o}lya's inequality holds on $\Omega_p$ for any $k\geq 2$. It also holds for $k=1$.
\end{proof}
\begin{remark}\label{remextension}
 The same proof may be used to obtain a similar result if $\Omega$ is a manifold with boundary -- see the examples in Sections~\ref{manifgeo} and~\ref{manifex}
 below, which also shows that the condition on the existence of the order $p_{0}$ cannot be dropped in general.
\end{remark}
\begin{remark}
 Clearly it is possible to state the above result for a specific value of $p_{0}$ for which the conditions are satisfied. However, and
 without the assumption that the tiling takes place for all $p$, it cannot be asserted that the conjecture is satisfied for all $p$
 larger than $p_{0}$, as it may happen that the domain $\Omega$ cannot be tiled for all such values of $p$. To see this, it is enough to
 consider a regular $n-$polygon which is tiled by $n$ equal triangles (which satisfy P\'{o}lya, as they are also tiling domains), but
 which, except for $4^kn$ or $2\times 4^kn$, with $k\in\N_{0}$, will not be tiled with similar triangles for other integers larger than $n$.
\end{remark}

\section{Examples}

We shall now use Lemma~\ref{mainlem} to provide some examples of non-tiling domains satisfying P\'{o}lya's conjecture.

\subsection{Sectors of domains of revolution}\label{revolution}
Let $\Omega$ be a convex domain of revolution around an axis in $\R^{n}$ and with analytic boundary. Then it satisfies
the nonperiodicity condition~\cite[Lemma 1.3.19]{sava}, and thus also P\'{o}lya's conjecture eventually. On the other
hand, $\Omega$ is clearly tiled by $p$ sectors of $\Omega$ with an angle opening $2\pi/p$, and Lemma~\ref{mainlem} may
then be applied to obtain that, for all $p$ sufficiently large, these sectors satisfy P\'{o}lya's conjecture.

In the particular case of an $n-$dimensional ball $B$, it can be $p-$tiled by circular sectors $S_{p}$ with the same radius
and angle opening $2\pi/p$, for any positive integer $p$,  as described in the previous example.
The proof that $D$ satisfies P\'{o}lya's conjecture allowed the authors in~\cite{flps} to deduce that the same will happen for any sector of this form
Although Lemma~\ref{mainlem} only yields the existence of sufficiently thin sectors for which P\'{o}lya's conjecture holds, say $S_{p_{0}}$, it does not require
the angle to be a rational multiple of $\pi$ and may be extended to sectors of shells.

More precisely, given any angle $\alpha$ in $(0,\pi]$, we may conclude that there exists $q=q_{\alpha}\in \N$ such that P\'{o}lya's conjecture holds for sectors
with opening angle $\alpha/(jq_{\alpha})$ for all $j\in\N$. That Weyl's two-term asymptotics~\eqref{2termweyl} are satisfied in these cases, stems from the fact that
such domains allow for separation of variables~\cite{netr}.

% Similar results may be obtained for planar annular sectors.

% By extending the argument to domains $p-$tiling a ball $B$ in $\R^{2n}$, such as sets of the form
% $ S_{p_1,\ldots, p_n} =C_{p_1}\times \ldots \times C_{p_n} \cap B$, 
% where for each $i$, $p_i\in \mathbb{N}$, $C_{p_i}$ is a sector of $\mathbb{R}^2$ of angle $2\pi/p_i$, 
% $C_{p_i}=\left \{ re^{i\varphi}: \varphi\in [0, {2\pi}/{p_i}], r\geq 0\right\}$, 
% and $p=p_1\times \ldots\times p_n$.
% For a ball $B$ in $\R^{2n+1}$, we may take  
% $S_{p_1,\ldots, p_n} =C_{p_1}\times \ldots C_{p_n}\times \mathbb{R}^* \cap B$, 
% where $\mathbb{R}^*$ can be either $\mathbb{R}$ or
% $\mathbb{R}^+$, giving a $p-$tilling in the first case
% and a $(p+1)-$tilling in the second one.

\subsection{Solid cylinders} Let $\Omega\subset\R^{n}$ be a domain satisfying the nonperiodicity condition (Definition 1.3.7
in~\cite{sava}), and consider the cylinder obtained by the cartesian product $\cdr=\Omega\times J\subset\R^{n+1}$, where $J$ is the interval
$[0,\ell]$. By considering the projection of the trajectories in $\cdr$ onto the basis $\Omega$ of the cylinder, it is not difficult to see that $\cdr$
also satisfies the nonperiodicity condition -- see also~\cite{mikh}, where the two-term asymptotics was established directly in the case when $n$ is two. By
Theorem~1.6.1 in~\cite{sava} (see also the Remark~1.6.2
following the theorem and Example 1.6.16), the cylinder $\cdr$ satisfies the two-term Weyl asymptotics~\eqref{2termweyl} and~\eqref{2termweylNeumann}, and
hence it satisfies P\'{o}lya's conjecture eventually. We thus obtain that for each of the Dirichlet and Neumann cases there exist values of $\ell_{p}$ 
sufficiently small such that the cylinder $\mathcal{C}_{\ell_{p}}$ satisfies P\'{o}lya's conjecture.\vspace*{5mm}

Lemma~\ref{mainlem} may also be applied non-Euclidean examples, where the main point is again whether or not the nonperiodicity condition is satisfied. However,
now there are known examples where it actually fails, such as hemispheres -- see~\cite[Example 1.3.16]{sava} and~\cite{fms}.

\subsection{Sectors of geodesic disks\label{manifgeo}}
Provided we can ensure both the nonblocking and nonperiodicity conditions, a similar approach to that of example~\ref{revolution} above may be used for sectors of
geodesic disks on manifolds. Concerning the former condition, this will hold for geodesic disks in both $\mathbb{S}^{n}$ and $\mathbb{H}^{n}$

The nonperiodicity condition will hold for strongly convex spherical caps of the standard $n-$dimensional sphere $\mathbb{S}^{n}$, that is, those that are
strictly contained in a hemisphere~\cite[Example 1.3.16]{sava}. We thus obtain that sufficiently thin sectors on $\mathbb{S}^{n}$ will satisfy P\'{o}lya's conjecture.
In the hyperbolic case, it is possible to derive from~\cite[Lemma 1.3.19]{sava} that the nonperiodicity condition holds for geodesic disks of any radius, and thus again the conjecture
will be satisfied for sufficiently thin sectors.

\subsection{Cylindrical surfaces\label{manifex}} This example shows that, in general, the condition that $p$ be large enough in Lemma~\ref{mainlem} cannot be dropped. Let 
$M$ be a closed $(n-1)-$manifold and assume the conditions in the lemma are satisfied for the domain $S_{\ell} = M\times J$, where again $J$ is the interval $[0,\ell]$ -- recall
that, as pointed out in Remark~\ref{remextension}, Lemma~\ref{mainlem} also holds in this case.
This means that $S_{\ell}$ may be sliced up into $p$ domains in the way described in the lemma, each being a cylindrical surface itself, of the form $S_{\ell/p}$. Then,
for large enough $p$, $S_{\ell/p}$ will satisfy P\'{o}lya's conjecture.

On the other hand, since $M$ is closed, its first eigenvalue is zero and the first eigenvalue of $S_{h}$ equals $\pi^{2}/h^{2}$. This eigenvalue will
satisfy P\'{o}lya's conjecture if and only if
\[
 \fr{\pi^2}{h^2} \geq \fr{4\pi^2}{\left(\omega_{n} h |M|\right)^{2/n}} \Leftrightarrow h \leq \fr{1}{2^{n/(n-1)}}\left( \omega_{n} |M| \right)^{1/(n-1)}.
\]
We thus see that, given a value of $\ell$, unless the height $h=\ell/p$ is small enough, the first eigenvalue will not satisfy the above inequality.

A specific example may be obtained by considering the two-dimensional cylindrical surface $S_{\ell}=\mathbb{S}^{1}\times [0,\ell]$, where $\mathbb{S}^{1}$ denotes the unit
circle. This satisfies the necessary conditions for the two-term Weyl asymptotics given by~\eqref{2termweyl} to hold, and thus, given large enough
$\ell$, $S_{\ell}$ does not satisfy P\'{o}lya's conjecture, while $S_{\ell/p}$ will do so, provided $p$ is a sufficiently large integer.

\section{Improved constants for Cylinders}
The results in~\cite{lapt} require the space dimension of the domain satisfying P\'{o}lya's conjecture to be at least two, and thus do not apply to
general cylinders of the form $\cdr = \Omega \times [0, \ell]\subset\R^{n+1}$, where $\Omega$ is a bounded domain in $\mathbb{R}^n$
and $\ell$ a positive real number -- note that the interval satisfies P\'{o}lya's inequality which in this case becomes an identity for all eigenvalues.
Thus, and to the best of our knowledge, the best result so far in this case is the Li and Yau bound~\eqref{polyaineq} with constant given by~$(n+1)/(n+3)$.
This is a consequence of a result for the sum of the first $k$ eigenvalues, namely, that Dirichlet eigenvalues of an Euclidean
domain $\Omega$ in $\R^{n}$ satisfy
\begin{eqnarray} \label{LiYauI}
\fr{1}{k}\left(\sum_{j=1}^k\lambda_j(\Omega)\right)\geq \beta_{n}\fr{4\pi^2k^{2/n}}{(\omega_{n}|\Omega|)^{2/n}},
\quad \mbox{with}\quad \beta_{n}:=\fr{n}{n+2}.
\end{eqnarray}
which is asymptotically sharp as $k$ goes to infinity -- as was mentioned in~\cite{lapt}, this inequality is also related to a previous result by Berezin~\cite{bere}.
In the next result we use~\eqref{LiYauI} to improve the constant $C_{n}$ in~\eqref{polyaineq} for general cylinders.

\begin{thm}\label{TheoB} Consider the cylinder $\cdr = \Omega \times [0, \ell]\subset\R^{n+1}$, where $\Omega$ is a bounded domain in $\mathbb{R}^n$,
and $\ell$ any positive real number. Then its Dirichlet eigenvalues satisfy
\begin{equation}\label{improve}
\lambda_{k}(\cdr)\geq \alpha_{n+1} \left(\fr{4\pi^2}{\omega_{n+1}^{2/(n+1)}}\right)\left(\fr{k}{\ell |\Omega|}\right)^{2/(n+1)}, 
\end{equation}
where 
\begin{equation}\label{Freitas-constant}
\alpha_{n}:= \fr{n}{(n+1)^{1-\frac{1}{n}}} \left(\frac{\pi }{4}\right)^{1/n} \left[\frac{\Gamma \left(\frac{n+1}{2}\right)}{\Gamma \left(\frac{n+2}{2}\right)}\right]^{2/n}.
\end{equation}
For any $n\geq 2$, the coefficients $\alpha_{n}$ satisfy
\[
\beta_{n+1} < \fr{n+1}{n+2}\left(\fr{\pi}{2}\times\fr{n+3}{n+2}\right)^{1/(n+1)} < \alpha_{n+1} <
\fr{n+1}{n+2}\left(\fr{\pi}{2}\times\fr{n+2}{n+1}\right)^{1/(n+1)} <1, %\beta_n^{n/(n+1)}
\] 
where $\beta_{n}$ denotes the Berezin-Li-Yau coefficients defined in~\eqref{LiYauI}.
 \end{thm}
 \begin{proof}
 Denoting the eigenvalues of $\Omega$ in increasing order and repeated according to multiplicities by $\eta_j$, and those of the interval $J=[0,\ell]$ by $\rho_{l} = \pi^2l^2/\ell^2$,
 the eigenvalues of the Cartesian product $\cdr=\Omega \times [0,\ell]$, are given by $\eta_j+\rho_l$. Given $\lambda >0$, let  $j_{\lambda}$ be the largest integer such that
$ \lambda -\eta_{j_{\lambda}}\geq 0$. Then we have
\begin{eqnarray}
\mathcal{N}_{\Omega \times J}(\lambda)&=& \#\{(j,l): \eta_j +\rho_l\leq \lambda\}\nonumber\eqskip
  & = &  \#\{(j,l): \rho_l \leq \lambda-\eta_j\} \nonumber\eqskip
  & = & \dsum_{j} \#\{l: \rho_l \leq (\lambda-\eta_j)_+\} \nonumber\eqskip
  & = & \dsum_{j=1}^{j_{\lambda}}\mathcal{N}_{J}((\lambda-\eta_j))\nonumber\eqskip
  & = & \dsum_{j=1}^{j_{\lambda}}\left\lfloor \fr{\ell}{\pi}\sqrt{\lambda -\eta_j}\right\rfloor\nonumber\eqskip
  & \leq &  \dsum_{j=1}^{j_{\lambda}} \fr{\ell}{\pi}\sqrt{\lambda -\eta_j}. \label{NMJ}
\end{eqnarray}
Using the Cauchy-Schwarz inequality we obtain
\begin{eqnarray}
\mathcal{N}_{\Omega \times J}(\lambda) & \leq & \fr{\ell}{\pi}
\sqrt{j_{\lambda}}\left( \sum_{j=1}^{j_{\lambda}}(\lambda -\eta_j)\right)^{1/2}\nonumber\eqskip
& = & \fr{\ell}{\pi} \sqrt{j_{\lambda}}\left( \lambda j_{\lambda}- \sum_{j=1}^{j_{\lambda}}\eta_j\right)^{1/2}\label{ineq1}.
\end{eqnarray}
The Berezin-Li-Yau inequality~\eqref{LiYauI} applied to the eigenvalues of $\Omega$ reads 
\[
 \sum_{j=1}^{k}\eta_j\geq \beta_n\,
\fr{4\pi^2 k^{2/n +1}}{(\omega_n |\Omega|)^{2/n}},
\]
and plugging this in~\eqref{ineq1} then yields
\begin{eqnarray*}
k:=\mathcal{N}_{\Omega \times J}(\lambda)&\leq& \fr{\ell j_{\lambda}}{\pi}
\left( \lambda -\beta_n \,\fr{4\pi^2j_{\lambda}^{2/n}}{(\omega_n |\Omega|)^{2/n}}\right)^{1/2}.
\end{eqnarray*}
Thus, for $\lambda=\lambda_k$, we have
\begin{equation}\label{KEY}
\lambda_k \geq \fr{\pi^2k^2}{\ell^2j_{\lambda}^2}+\beta_n \fr{4\pi^2j_{\lambda}^{2/n}}{(\omega_n|\Omega|)^{2/n}}.
\end{equation}
The right-hand side may now be viewed as a function on $j_{\lambda}^2$, having its minimum at
\[
 j_{\lambda}^2=\left(\fr{n k^2}{4\ell^2 \beta_n}\right)^{n/(n+1)} (\omega_n|\Omega|)^{2/(n+1)}.
\]
Plugging this back in~\eqref{KEY} we obtain the following lower bound for $\lambda_{k}$
\begin{eqnarray*}
\lambda_{k} &\geq & \fr{\pi^2 k^2}{\ell^2} \left(\fr{4\beta_n\ell^2}{n k^2}\right)^{n/(n+1)}\fr{1}{\left(\omega_{n}|\Omega|\right)^{2/(n+1)}}\eqskip
& & \hspace*{5mm}+\ 4\pi^2 \beta_{n} \left(\fr{n k^{2}}{4\ell^{2}\beta_{n}}\right)^{1/(n+1)}\fr{1}{\left(\omega_{n}|\Omega|\right)^{2/(n+1)}}\eqskip
&=& \fr{\left(4\beta_{n}\right)^{n/(n+1)}\pi^2}{\left( \omega_{n}\ell |\Omega|\right)^{2/(n+1)}}\left( \fr{1}{n^{n/(n+1)}} + n^{1/(n+1)} \right)
k^{2/(n+1)}.
\end{eqnarray*}
Finally, since the cylinder is now in $\R^{n+1}$, and recalling that
\[
 \omega_{n+1}=\omega_{n}\sqrt{\pi}\, \fr{\Gamma\left(\fr{n}{2}+1\right)}{\Gamma\left(\fr{(n+1)}{2}+1\right)},
\]
we obtain, by replacing for $\omega_{n}$ in the previous expression,
\begin{eqnarray*}
\lambda_k & \geq & \fr{\left(4\beta_{n}\right)^{n/(n+1)}\pi^2}{\left( \omega_{n+1}\ell |\Omega|\right)^{2/(n+1)}}
\left[\sqrt{\pi}\, \fr{\Gamma\left(\fr{n}{2}+1\right)}{\Gamma\left(\fr{(n+1)}{2}+1\right)}\right]^{2/(n+1)}
\left( \fr{1}{n^{n/(n+1)}} + n^{1/(n+1)} \right)
k^{2/(n+1)}\eqskip
& = & \fr{n+1}{(n+2)^{1-1/(n+1)}}\left(\fr{\pi}{4}\right)^{1/(n+1)}
\left[\fr{\Gamma\left(\fr{n}{2}+1\right)}{\Gamma\left(\fr{(n+1)}{2}+1\right)}\right]^{2/(n+1)}
\fr{4\pi^2}{\left( \omega_{n+1}\ell |\Omega|\right)^{2/(n+1)}}\ k^{2/(n+1)}\eqskip
& = & \alpha_{n+1} \fr{4\pi^2}{\left( \omega_{n+1}\ell |\Omega|\right)^{2/(n+1)}}\ k^{2/(n+1)}
\end{eqnarray*}
as desired.

It remains to prove the inequalities in the second part of the theorem. The left- and right-most inequalities follow in a straightforward way by algebraic manipulation.
To prove the upper and lower bounds for the coefficients $\alpha_{n}$, we make use of Wendel's inequalities for quotients of gamma functions~\cite{wend}, which in this case read as
\[
 \fr{1}{\left( \fr{n+1}{2}\right)^{1/2}} \leq \fr{\Gamma\left(\fr{n+1}{2}\right)}{\Gamma\left(\fr{n}{2}+1\right)} \leq \fr{ \left(\fr{n}{2}+1\right)^{1/2}}{\fr{n+1}{2}}.
\]
Replacing these in the expression for $\alpha_{n}$ yields the two inequalities.
 \end{proof}
 
% \begin{remark}
%  The fact that $\Omega$ is an Euclidean domain is not important here, as long as inequality~\eqref{LiYauI} holds, in which case the proof may then be extended to
%  the Cartesian product of a manifold with an interval. 
% \end{remark}

\begin{remark}
 The first two terms in the asymptotic behaviour of the coefficients $\alpha_{n}$ defined in Theorem~\ref{TheoB} are given by $1-\fr{1-\log(\pi/2)}{n}$, which may still
 be seen to be smaller than $\alpha_{n+1}$, while remaining larger than $\beta_{n+1}$.
\end{remark}

 \begin{remark}
  As we see from the last part of the result, our constant is larger than the corresponding Berezin-Li-Yau constant $\beta_{n}$. For
  instance, when $n=2$ we obtain $\alpha_{3}= (3/4)^{1/3} \approx 0.90856$ while the corresponding $\beta_{3}$ equals $3/5$.
\end{remark}

\begin{remark}
The essential property of the domain $\Omega$ which is used in the proof of Theorem~\ref{TheoB} is that it satisfies
the Berezin-Li-Yau estimate~\eqref{LiYauI}. Provided this is the case, a similar proof may be used to obtain the same result
when $\Omega$ is a domain (with boundary) on a manifold.
\end{remark}
 
\section*{Acknowledgements} We would like to thank Volodymyr Mikhailets for pointing out reference~\cite{mikh} to us.
This work was partially supported by the Funda\c c\~{a}o para a Ci\^{e}ncia e a Tecnologia (Portugal) through project UIDB/00208/2020.


\begin{thebibliography}{99999}

\bibitem[B]{bere}
F.A. Berezin, Covariant and contravariant symbols of operators,
{\it Izv. Akad. Nauk SSSR Ser. Mat.} {\bf 13}, 1134--1167 (1972).

\bibitem[CS]{colbsouf}
B. Colbois and A. El Soufi,
Extremal eigenvalues of the Laplacian on Euclidean domains and closed surfaces,
\emph{Math. Z.} {\bf 278} (2014), 529--549.

\bibitem[FLPS]{flps}
N.D. Filonov, M. Levitin, I. Polterovich and D. A. Sher,
P\'{o}lya's conjecture for Euclidean balls,
{\it Invent. Math.} (2023).

\bibitem[F]{frei1}
P. Freitas,
Upper and lower bounds for the first Dirichlet eigenvalue of a triangle,
{\it Proc. Amer. Math. Soc.} {\bf 134} (2006), 2083--2089.

% \bibitem[[F2]{frei2}
% P. Freitas,
% A remark on P\'olya's conjecture at low frequencies.
% Arch. Math. (Basel)  {\bf 112} (2019), 305--311.

% \bibitem[FK]{frke}
% P. Freitas and J. B. Kennedy,
% Extremal domains and P\'{o}lya-type inequalities for the Robin Laplacian on rectangles and unions of rectangles,
% Int. Math. Res. Not. IMRN 2021, no. 18, 13730--13782.

\bibitem[FLP]{frlapa}
P. Freitas, J. Lagac\'{e} and J. Payette,
Optimal unions of scaled copies of domains and P\'olya's conjecture,
{\it Ark. Math.} {\bf 59} (2021),11--51.

\bibitem[FMS]{fms}
P. Freitas, J. Mao and I. Salavessa,
P\'{o}lya-type inequalities on spheres and hemispheres, preprint.

\bibitem[HS]{hs}
E.M. Harrell and J. Stubbe,
Two-term, asymptotically sharp estimates for eigenvalue means of the Laplacian,
{\it J. Spectr. Theory} {\bf 8} (2018), 1529--1550.
Partial retraction DOI 104171/JST/353 

\bibitem[H]{hers}
J. Hersch,
Bounds for eigenvalues of P\'{o}lya's ``plane covering domains'' by filling
a torus or a cylinder,
{\it J. d'Anal. Math.} {\bf 30} (1976), 265--270.

% \bibitem[J]{jame}
% G.J.O. Jameson,
% Inequalities for gamma function ratios,
% Amer. Math. Monthly {\bf 120} (2013), 936--940.

\bibitem[Ke]{kelln}
R. Kellner,
On a theorem of Polya,
{\it Amer.\ Math.\ Monthly} {\bf 73} (1966), 856--858.

\bibitem[Kr]{kroger}
P. Kr\"{o}ger,
Upper bounds for the Neumann eigenvalues on a bounded domain in Euclidean space.
{\it J. Funct. Anal.} {\bf 106} (1992), 353--357.

\bibitem[L]{lapt}
A. Laptev, 
Dirichlet and Neumann eigenvalue problems on domains in Euclidean space.
{\it J. Funct. Anal.} {\bf 131} (1997), 531--545.

% \bibitem[LPS]{LPS}
% ,M. Levitin, I. Polterovich and D. A. Sher
% P\'{o}lya's conjecture for the disk: a computer-assisted proof,
% preprint arxiv no. 2203.07696.

\bibitem[LY]{liyau}
P. Li and S.-T. Yau,
On the Schr\"odinger equation and the eigenvalue problem.
{\it Comm. Math. Phys.} {\bf 88} (1983), 309--318.

\bibitem[L]{lieb}
E. Lieb, The number of bound states of one-body Schrodinger operators and the Weyl problem.
{\it Proc. Sym. Pure Math.} {\bf 36} 241--252 (1980).

\bibitem[M]{mikh}
V.A. Mikhailets,
Distribution of the eigenvalues of the Sturm--Liouville operator equation,
{\it Izv. Akad. Nauk SSSR Ser. Mat} {\bf 41} (1977), 607--619; English transl. in
{\it Math. USSR-Izv.} {\bf 41} (1977), 571--582.

\bibitem[N]{netr}
N.V. Kuznecov,
Asymptotic distribution of eigenfrequencies of a plane membrane in the case of separable variables. (Russian)
{\it Differencial'nye Uravnenija} {\bf 2} (1966), 1385--1402. 

\bibitem[P1]{poly1}
G. P\'{o}lya,
{\it Mathematics and plausible reasoning: patterns of plausible inference}, 2nd Edition,
Princeton University Press (1968). % first edition 1954

\bibitem[P2]{poly2}
G. P\'{o}lya,
On the eigenvalues of vibrating membranes.
{\it Proc. London Math. Soc.} {\bf 11} (1961), 419--433.

\bibitem[U]{urak83}
H. Urakawa,
Lower bounds for the eigenvalues of the fixed vibrating membrane problems,
{\it T\^{o}hoku Math. J.} {\bf 36} (1983), 185--189.

\bibitem[SV]{sava}
Yu. Safarov and D. Vassiliev,
{\em The asymptotic distribution of eigenvalues of partial differential operators},
American Mathematical Society, series Translations of Mathematical Monographs, {\bf 155}, 1997.

\bibitem[W]{wend}
J.G. Wendel,
Note on the Gamma Function,
{\it Amer. Math. Monthly} {\bf 55} (1948), 563--564.

\end{thebibliography}
\end{document}